\documentclass[11pt,times,twoside]{article}

\usepackage[body={14.5cm,22cm}, left=2.5cm, right=2.5cm, top=3cm]{geometry}
\usepackage{graphicx,amsmath,amsthm,latexsym,amssymb,amstext}
\usepackage[pdfstartview=FitH,
bookmarksnumbered=true,bookmarksopen=true,%
colorlinks=true,pdfborder=001,citecolor=blue,%
linkcolor=blue,urlcolor=blue]{hyperref}

\newtheorem{theorem}{Theorem}[section]
\newtheorem{definition}[theorem]{Definition}

\newtheorem{proposition}[theorem]{Proposition}
\newtheorem{lemma}[theorem]{Lemma}

\newtheorem{remark}[theorem]{Remark}
\newtheorem{example}{Example}
\numberwithin{equation}{section}

\begin{document}

\title{\Large \bf Mean-field backward stochastic differential equations driven by fractional Brownian motion }

\author
{\textbf{Jiaqiang Wen}$^{1,3}$, \textbf{Yufeng Shi}$^{1,2,}$\thanks{Corresponding author.
  E-mail addresses: jqwen59@gmail.com (J. Wen), yfshi@sdu.edu.cn (Y. Shi)} \\
\normalsize{$^{1}$Institute for Financial Studies, Shandong University, Jinan 250100, China}\\
\normalsize{$^{2}$School of Statistics, Shandong University of Finance and Economics, Jinan 250014, China}\\
\normalsize{$^{3}${Department of Mathematics, University of Central Florida, Orlando, FL 32816, USA}}\\
}

\date{}

\renewcommand{\thefootnote}{\fnsymbol{footnote}}

\footnotetext[0]{This work is supported by
     NNSF of China (Grant Nos. 11371226, 11526205, 11626247),
     the Foundation for Innovative Research Groups of National Natural Science Foundation  of China (Grant No. 11221061) and the 111 Project (Grant No. B12023).}

\maketitle

\begin{abstract}
In this paper, we focus on the mean-field backward stochastic differential equations (BSDEs) driven by
a fractional Brownian motion with Hurst parameter $H>\frac{1}{2}$.
First, the existence and uniqueness of these equations are established under Lipschitz condition.
Then, a comparison theorem for such mean-field BSDEs is obtained.
Finally, as an application, we connect this mean-field BSDE with a nonlocal partial differential equation (PDE).
\end{abstract}



\section{Introduction}

A centered Gaussian process $B^{H}=\{ B^{H}_{t},t\geq 0 \}$ is called a
fractional Brownian motion (fBm, for short) with Hurst parameter $H \in (0, 1)$ if its covariance is
\begin{equation*}
 E(B^{H}_{t}B^{H}_{s}) = \frac{1}{2} (t^{2H} + s^{2H} - |t-s|^{2H}), \ \ t,s \geq 0.
\end{equation*}
When $H=\frac{1}{2}$, this process becomes a classical Brownian motion.
For $H>\frac{1}{2}$, $B^{H}_{t}$ exhibits the property of long range dependence,
which makes fBm an important driving noise in many fields such as finance, telecommunication networks, and physics.

In 1990, the nonlinear backward stochastic differential equation (BSDE, for short) was introduced by
 Pardoux and Peng \cite{Peng}.
In the next two decades, it has been widely used in different fields of mathematical finance \cite{Peng2}, stochastic control \cite{Yong5}, and partial differential equations (PDEs, for short) \cite{Peng92}.
Recently, Buckdahn et al. \cite{Buckdahn} and Buckdahn, Li and Peng \cite{Buckdahn2} introduced the so-called mean-field BSDEs,
owing to that mathematical mean-field approaches play an important role in many fields,
such as  Economics, Physics and Game Theory
(see Lasry and Lions \cite{Lasry}, Buckdahn et al. \cite{Buckdahn2017} and the papers therein).
Furthermore, BSDEs driven by fBm, also known as the fractional BSDEs,
with Hurst parameter $H>\frac{1}{2}$ were studied by Hu and Peng \cite{Hu}.
Then Maticiuc and Nie \cite{Maticiuc} obtained some general results of the fractional BSDEs through a rigorous approach.
 Buckdahn and Jing \cite{Buckdahn3} studied the fractional mean-field stochastic differential equations (SDEs, for short) with $H>\frac{1}{2}$ and a stochastic control problem.
Some other recent developments of fractional BSDEs  can be found in
Bender \cite{Bender}, Borkowska \cite{Borkowska}, Maticiuc and Nie \cite{Maticiuc},
 Wen and Shi \cite{Wen}, etc.

Motivated by the above works, in this paper,
the purpose is to investigate the following fractional mean-field BSDEs:
\begin{equation}\label{0}
  Y_{t} = \xi + \int_t^T E'[f(s,\eta_{s},Y'_{s},Z'_{s},Y_{s},Z_{s})] ds - \int_t^T Z_{s} dB_{s}^{H}, \ \ 0\leq t\leq T,
\end{equation}
where Hurst parameter $H>\frac{1}{2}$, and
the stochastic integral in (\ref{0}) is the divergence type integral
(see Decreusefond and \"{U}st\"{u}nel \cite{Decreusefond}, and  Nualart \cite{Nualart}).
First, two different methods are proposed to prove the existence and uniqueness of Eq. (\ref{0}).
Interestingly,  the conditions required by the first method are weaker, however, the second method is more convenient.
Then, for its wide applications to BSDEs, a comparison theorem of such mean-field BSDEs is obtained.
In addition, we connect this mean-field BSDE with a nonlocal PDE.
It should be pointed out that motivated by Biagini et al. \cite{Biagini02} and Han, Hu and Song \cite{Han13},
there may exist the dual relation between mean-field SDEs and mean-field BSDEs in the circumstance of fractional calculus.
Moreover, similar to mean-field SDEs driven by fBm introduced in \cite{Buckdahn3},
our mean-field BSDEs can also be applied to the field of stochastic optimal controls.
About these topics, some further studies in the coming future researches will be given.

We organize this article as follows.
Some preliminaries about fBm are presented in Section 2.
The existence and uniqueness of the fractional mean-field BSDEs are proved by two different methods in Section 3.
We derive a comparison theorem for such mean-field BSDE in Section 4,
and connect this mean-field BSDE with a nonlocal PDE in Section 5.

\section{Preliminaries}

We recall, in this section, some basic results of fractional Brownian motion.
For a deeper discussion, the readers may refer to the articles such as Decreusefond and \"{U}st\"{u}nel \cite{Decreusefond},
Hu \cite{Hu3} and  Nualart \cite{Nualart}.

Assume $B^{H}=\{ B^{H}_{t},t\geq 0 \}$ is a fBm defined on a complete probability space $(\Omega,\mathcal{F},P)$,
and the filtration $\mathcal{F}$ is generated by $B^{H}$.
Let $H >\frac{1}{2}$ throughout this paper.
Moreover, we denote $\phi(x) = H(2H - 1)|x|^{2H-2},$ where $x \in \mathbb{R}$,
and suppose $\xi$ and $\eta$ are two continuous functions defined in $[0,T]$.
Define
\begin{equation}\label{11}
  \langle \xi,\eta \rangle_{T} = \int_0^T \int_0^T \phi(u-v) \xi_{u} \eta_{v} dudv, \ \
  and \ \   \| \xi \|_{T}^{2} =  \langle \xi,\xi  \rangle_{T}.
\end{equation}
Then, $\langle \xi,\eta \rangle_{T}$ is a Hilbert scalar product.
Under this scalar product, we denote by $\mathcal{H}$ the completion of the continuous functions.
Besides, denote by $\mathcal{P}_{T}$ the set of all polynomials of fBm in $[0,T]$, i.e.,
every element of $\mathcal{P}_{T}$ is of the form
\begin{equation*}
  \Phi(\omega) = h \left(\int_0^T \xi_{1}(t) dB_{t}^{H},...,\int_0^T \xi_{n}(t) dB_{t}^{H} \right),
\end{equation*}
where $h$ is a polynomial function and $\xi_{i}\in\mathcal{H}, i=1,2,...,n$.
In addition,
Malliavin derivative operator $D_{s}^{H}$ of $\Phi\in \mathcal{P}_{T}$ is defined by:
\begin{equation*}
  D_{s}^{H}\Phi = \sum\limits_{i=1}^{n} \frac{\partial h}{\partial x_{i}}
               \left(\int_0^T \xi_{1}(t) dB_{t}^{H},...,\int_0^T \xi_{n}(t) dB_{t}^{H} \right)\xi_{i}(s), \ \ s\in [0,T].
\end{equation*}
Since the derivative operator
  $D^{H}:L^{2}(\Omega,\mathcal{F}, P) \rightarrow (\Omega,\mathcal{F}, \mathcal{H})$ is closable,
one can denote by $\mathbb{D}^{1,2}$ the completion of $\mathcal{P}_{T}$ under the following norm
$$\| \Phi \|^{2}_{1,2} = E|\Phi|^{2} + E\|D^{H}_{s} \Phi\|^{2}_{T}.$$
Furthermore, we introduce the following derivative
\begin{equation}\label{12}
  \mathbb{D}_{t}^{H}\Phi = \int_0^T \phi(t-s) D_{s}^{H}\Phi ds, \ \ t\in[0,T].
\end{equation}
Now, let us consider the adjoint operator of Malliavin derivative operator $D^{H}$.
We call this operator the divergence operator, which represents the divergence type integral and
 is denoted by $\delta(\cdot)$.

\begin{definition}
A process $u\in L^{2}(\Omega\times[0,T];\mathcal{H})$ is said to belongs to the domain $Dom(\delta)$,
if there exists $\delta(u)\in L^{2}(\Omega,\mathcal{F},P)$ satisfying the following duality relationship:
\begin{equation*}\label{}
  E(\Phi\delta(u))=E(\langle D^{H}_{\cdot} \Phi,u \rangle_{T}), \ \ for \ every \  \Phi\in\mathcal{P}_{T}.
\end{equation*}
Moreover, if $u\in Dom(\delta)$, the divergence type integral of $u$ w.r.t. $B^{H}$ is defined by putting
$\int_0^T u_{s} d B^{H}_{s}=: \delta(u)$.
\end{definition}
It should be pointed out that, in this paper, unless otherwise specified,
 the $d B^{H}$-integral represents the divergence type integral.

\begin{proposition}[Hu \cite{Hu3}, Proposition 6.25]\label{2}
Let $\mathbb{L}^{1,2}_{H}$ be the space of all processes $F : \Omega\times[0,T] \rightarrow \mathcal{H}$ satisfying
$ E \left( \| F \|_{T}^{2} + \int_0^T \int_0^T |\mathbb{D}_{s}^{H}F_{t}|^{2} dsdt \right) < \infty.$
Then, if $F \in \mathbb{L}^{1,2}_{H}$, the divergence type integral
$\int_0^T F_{s} dB_{s}^{H}$ exists in $L^{2}(\Omega,\mathcal{F}, P)$, and
\begin{equation*}
  E \left( \int_0^T F_{s} dB_{s}^{H} \right) = 0; \ \
   E \left( \int_0^T F_{s} dB_{s}^{H} \right)^{2}
 =E \left( \| F \|_{T}^{2} + \int_0^T \int_0^T \mathbb{D}_{s}^{H}F_{t} \mathbb{D}_{t}^{H}F_{s} dsdt \right).
\end{equation*}
\end{proposition}

\begin{proposition}[Hu \cite{Hu3}, Theorem 10.3]\label{4}
Suppose  $g$ and  $f$ are two deterministic continuous functions. Let
\begin{equation*}
  X_{t} = X_{0} + \int_0^t g_{s} ds + \int_0^t f_{s} dB_{s}^{H}, \ \ t\in [0,T],
\end{equation*}
where $X_{0}$ is a constant. Then, if $F \in C^{1,2}([0, T ] \times \mathbb{R})$, one has
\begin{align*}
 F(t,X_{t})=& F(0,X_{0})+ \int_0^t \frac{\partial F}{\partial s}(s,X_{s}) ds
  + \int_0^t \frac{\partial F}{\partial x}(s,X_{s})g_{s} ds \\
 & + \int_0^t \frac{\partial F}{\partial x}(s,X_{s}) f_{s} dB_{s}^{H}
+ \frac{1}{2}\int_0^t \frac{\partial^{2} F}{\partial x^{2}}(s,X_{s}) \bigg[\frac{d}{ds} \| f \|_{s}^{2}\bigg] ds,
  \ \  t\in [0,T].
\end{align*}
\end{proposition}

\begin{proposition}[Hu \cite{Hu3}, Theorem 11.1]\label{3}
 For $i = 1, 2$, let $g_{i}$ and $f_{i}$ be two real valued processes satisfying
 $E \int_0^T (|g_{i}(s)|^{2} + |f_{i}(s)|^{2}) ds < \infty$.
Moreover, assume that $D^{H}_{t}f_{i}(s)$ is continuously differentiable in its arguments
 $(s,t)\in [0,T]^{2}$ for almost every $\omega \in \Omega$, and
$E \int_0^T \int_0^T |\mathbb{D}_{t}^{H}f_{i}(s)|^{2} dsdt < \infty$.
Denote
\begin{equation*}
  X_{i}(t) = \int_0^t g_{i}(s) ds + \int_0^t f_{i}(s) dB_{s}^{H}, \ \ t\in [0,T].
\end{equation*}
Then
\begin{equation*}
\begin{split}
   X_{1}(t)X_{2}(t) =& \int_0^t X_{1}(s)g_{2}(s) ds + \int_0^t X_{1}(s)f_{2}(s) dB_{s}^{H}
                      +\int_0^t X_{2}(s)g_{1}(s) ds \\
                    &+ \int_0^t X_{2}(s)f_{1}(s) dB_{s}^{H} +\int_0^t \mathbb{D}_{s}^{H}X_{1}(s)g_{2}(s) ds + \int_0^t \mathbb{D}_{s}^{H}X_{2}(s)g_{1}(s) ds.
\end{split}
\end{equation*}
\end{proposition}

\section{Existence and uniqueness theorem}

The existence and uniqueness of the fractional mean-field BSDEs are proved here.
For simplify the presentation, we only discuss the one dimensional case.
Let
\begin{equation}\label{33}
\eta_{t} = \eta_{0} + \int_0^t b_{s} ds + \int_0^t \sigma_{s} dB_{s}^{H},
\end{equation}
where $\eta_{0}$ is a constant, and $b$ and $\sigma$ are two deterministic differentiable functions,
such that $\sigma_{t} \neq 0$ (then either $\sigma_{t} < 0$ or $\sigma_{t} > 0$),  $t\in [0,T]$.
We recall that (see (\ref{11}))
\begin{equation*}
  \| \sigma \|_{t}^{2} = H(2H-1) \int_0^t \int_0^t |u-v|^{2H-2} \sigma_{u} \sigma_{v} dudv.
\end{equation*}
So $\frac{d}{dt}(\| \sigma \|_{t}^{2})= 2 \hat{\sigma}_{t} \sigma_{t}>0$ for $t\in (0,T]$, where
$\hat{\sigma}_{t} = \int_0^t \phi(t-v) \sigma_{v} dv$.

Now,
we denote the (non-completed) product space of $(\Omega,\mathcal{F},P)$ by $(\bar{\Omega},\bar{\mathcal{F}}, \bar{P})$ $ = (\Omega\times \Omega,\mathcal{F}\otimes \mathcal{F}, P\otimes P)$,
and denote the filtration of this product space by $\bar{\mathbb{F}} = \{ \bar{\mathcal{F}}_{t} = \mathcal{F} \otimes \mathcal{F}_{t}, 0\leq t\leq T \}$.
A random variable, originally defined on $\Omega$,
$\xi\in L^{0}(\Omega,\mathcal{F}, P;\mathbb{R})$ is canonically extended to $\bar{\Omega}$:
$\xi'(\omega',\omega)=\xi(\omega'), \ \ (\omega',\omega)\in \bar{\Omega} = \Omega\times \Omega$.
On the other hand, for every $\theta\in L^{1}(\bar{\Omega},\bar{\mathcal{F}}, \bar{P})$, the random variable
$\theta(\cdot,\omega):\Omega\rightarrow \mathbb{R}$ is in $L^{1}(\Omega,\mathcal{F}, P), \  P(d \omega), \ a.s.$,
and its expectation is denoted by
\begin{equation*}
  E'[\theta(\cdot,\omega)] = \int_{\Omega} \theta(\omega',\omega) P(d \omega').
\end{equation*}
Then we have $E'[\theta]=E'[\theta(\cdot,\omega)]\in L^{1}(\Omega,\mathcal{F}, P)$. In addition,
\begin{equation*}
 \bar{E}[\theta]\bigg(= \int_{\bar{\Omega}} \theta d \bar{P} = \int_{\Omega}  E'[\theta(\cdot,\omega)] P(d \omega) \bigg) = E\big[E'[\theta]\big].
\end{equation*}

Motivated by Buckdahn et al. \cite{Buckdahn,Buckdahn2}, we investigate the mean-field BSDEs driven by fBm as follows:
\begin{equation}\label{31}
  Y_{t} = \xi + \int_t^T E'[f(s,\eta_{s},Y'_{s},Z'_{s},Y_{s},Z_{s})] ds - \int_t^T Z_{s} dB_{s}^{H}, \ \ 0\leq t\leq T.
\end{equation}

\begin{remark}
Owing to our notation, we mark that the coefficient of BSDE (\ref{31}) is explained by:
\begin{equation*}
\begin{split}
     E'[f(s,\eta_{s},Y'_{s},Z'_{s},Y_{s},Z_{s})](\omega)
   =& E'[f(s,\eta_{s}(\omega),Y'_{s},Z'_{s},Y_{s}(\omega),Z_{s}(\omega))]\\
   =& \int_{\Omega} f(s,\eta_{s}(\omega),Y_{s}(\omega'),Z_{s}(\omega'),Y_{s}(\omega),Z_{s}(\omega)) P (d\omega').
\end{split}
\end{equation*}
\end{remark}
From the above remark, combining the definition of expectation, we have the following two special cases:
\begin{equation} \label{34}
 E'[f(s,Y'_{s},Z'_{s})]=E[f(s,Y_{s},Z_{s})], \ \  E'[f(s,\eta_{s},Y_{s},Z_{s})]=f(s,\eta_{s},Y_{s},Z_{s}).
\end{equation}

Before giving the definition of the solutions of BSDE (\ref{31}),
we introduce the following sets:

\begin{itemize}
  \item [$\bullet$]
   $C_{pol}^{1,3}([0,T]\times \mathbb{R})=\bigg\{ \varphi\in C^{1,3}([0, T] \times \mathbb{R}), \ and \
    \ all \ the \ derivatives \ of \  \varphi \ are \ polynomial \ growth \bigg\};$
  \item [$\bullet$]
  $\mathcal{V}_{[0,T]} = \bigg \{ Y=\varphi\big(\cdot,\eta(\cdot)\big) \mid \varphi \in C_{pol}^{1,3}([0,T]\times \mathbb{R})
  \  with \ \frac{\partial \varphi}{\partial t}  \in C_{pol}^{0,1}([0,T]\times \mathbb{R}), \ t\in[0,T] \bigg \}.$
\end{itemize}
 Moreover, by $\widetilde{\mathcal{V}}_{[0,T]}$ and $\widetilde{\mathcal{V}}_{[0,T]}^{H}$ denote the completion of $\mathcal{V}_{[0,T]}$ under the following norm respectively,
\begin{equation*}
  \| Y \| := \bigg(E\int_0^{T}  e^{\beta t}  |Y(t)|^{2} dt\bigg)^{\frac{1}{2}}, \ \ \ \
  \| Z \| := \bigg(E\int_0^{T} t^{2H-1} e^{\beta t} |Z(t)|^{2} dt\bigg)^{\frac{1}{2}},
\end{equation*}
where $\beta\geq 0$ is a constant.
It is easy to know  $\widetilde{\mathcal{V}}_{[0,T]}^{H} \subseteq \widetilde{\mathcal{V}}_{[0,T]} \subseteq L^{2}_{\mathcal{F}}(0,T;\mathbb{R})$.

\begin{definition}
We call $(Y, Z)$ a solution of BSDE (\ref{31}),
if they satisfy the following conditions:
\begin{itemize}
  \item [(i)]  $(Y,Z) \in \widetilde{\mathcal{V}}_{[0,T]}\times \widetilde{\mathcal{V}}^{H}_{[0,T]}$;
  \item [(ii)]
   $Y_{t} = \xi + \int_t^T E'[f(s,\eta_{s},Y'_{s},Z'_{s},Y_{s},Z_{s})] ds
          - \int_t^T Z_{s} dB_{s}^{H}, \  a.s.,  \ 0\leq t\leq T.$
\end{itemize}
\end{definition}

Next, we shall propose two different methods to prove the existence and uniqueness of Eq. (\ref{31}).

\subsection{The First Method}

In this subsection, the first method, introduced by Maticiuc and Nie \cite{Maticiuc}, is used to establish the existence and uniqueness of  Eq. (\ref{31}).
In order to find the solution of BSDE (\ref{31}), the following assumptions are needed.

\begin{itemize}
\item[(H1)] Suppose that $\xi = g(\eta_{T})$,
where $g\in C_{pol}^{1}( \mathbb{R})$;
\end{itemize}

\begin{itemize}
\item[(H2)]
Assume that for the coefficient $f=f(t,x,y',z',y,z):[0,T]\times \mathbb{R}^{5}\rightarrow \mathbb{R}$ with
 $f\in C_{pol}^{0,1}([0,T]\times \mathbb{R}^{5})$,  there is a constant $C\geq 0$ such that,
     for every $t\in [0,T]$, $x,y_{1},y_{2},z_{1},z_{2},y'_{1},$ $y'_{2},z'_{1},z'_{2} \in \mathbb{R}$, we have
\begin{align*}
      &|f(t,x,y'_{1},z'_{1},y_{1},z_{1}) - f(t,x,y'_{2},z'_{2},y_{2},z_{2})|\\
     &\leq C\big(|y'_{1}-y'_{2}| +|z'_{1}-z'_{2}| + |y_{1}-y_{2}| +|z_{1}-z_{2}| \big).
\end{align*}
\end{itemize}
For notational simplicity, we denote $f_{0}(t,x)=f_{0}(t,x,0,0,0,0)$.
\begin{theorem}
Under (H1) and (H2),  BSDE (\ref{31}) admits a unique solution.
  Moreover,
\begin{equation}\label{20}
      E \left( e^{\beta t}|Y_{t}|^{2} + \int_t^T e^{\beta s}s^{2H-1}|Z_{s}|^{2} ds\right)
 \leq R \Theta(t,T,K), \ \ t\in[0,T],
\end{equation}
where $R$ is a positive constant which may change line to line, and
\begin{equation*}
  \Theta(t,T,K) = E\bigg(e^{\beta T}|g(\eta_{T})|^{2}  + \int_t^T e^{\beta s}|f_{0}(s,\eta_{s})|^{2} ds \bigg).
\end{equation*}
\end{theorem}

\begin{proof}
For every $(y_{\cdot},z_{\cdot}) \in \widetilde{\mathcal{V}}_{[0,T]} \times \widetilde{\mathcal{V}}^{H}_{[0,T]}$,
consider the following simple BSDE:
\begin{equation}\label{32}
   Y_{t} = g(\eta_{T}) + \int_t^T E'[f(s,\eta_{s},y'_{s},z'_{s},y_{s},z_{s})] ds - \int_t^T Z_{s} dB_{s}^{H}, \ \ 0\leq t\leq T.
\end{equation}
From Proposition 17 of Maticiuc and Nie \cite{Maticiuc}, we know BSDE (\ref{32}) has a unique solution
$(Y_{\cdot},Z_{\cdot}) \in \widetilde{\mathcal{V}}_{[0,T]} \times \widetilde{\mathcal{V}}^{H}_{[0,T]}$.
Now define a mapping
$I:\widetilde{\mathcal{V}}_{[0,T]} \times \widetilde{\mathcal{V}}^{H}_{[0,T]}\longrightarrow \widetilde{\mathcal{V}}_{[0,T]} \times \widetilde{\mathcal{V}}^{H}_{[0,T]}$
such that $I[(y_{\cdot},z_{\cdot})]=(Y_{\cdot},Z_{\cdot})$.
Let $n\in \mathbb{N}$ and $t_{i}=\frac{i-1}{n}T, i=1,...,n+1$.
First we shall solve (\ref{31}) in $[t_{n},T]$.
In order to do this, we show $I$ is a contraction on $\widetilde{\mathcal{V}}_{[t_{n},T]} \times \widetilde{\mathcal{V}}^{H}_{[t_{n},T]}$.

For two arbitrary given elements $(y_{\cdot},z_{\cdot})$ and
 $(\overline{y}_{\cdot},\overline{z}_{\cdot})\in \widetilde{\mathcal{V}}_{[t_{n},T]} \times \widetilde{\mathcal{V}}^{H}_{[t_{n},T]}$,
let $(Y_{\cdot},Z_{\cdot})=I[(y_{\cdot},z_{\cdot})]$ and
$(\overline{Y}_{\cdot},\overline{Z}_{\cdot})=I[(\overline{y}_{\cdot},\overline{z}_{\cdot})]$.
We denote their differences  by
$(\hat{y}_{\cdot},\hat{z}_{\cdot})=(y_{\cdot}-\overline{y}_{\cdot},z_{\cdot}-\overline{z}_{\cdot})$ and
$(\hat{Y}_{\cdot},\hat{Z}_{\cdot})=(Y_{\cdot}-\overline{Y}_{\cdot},Z_{\cdot}-\overline{Z}_{\cdot})$.

By applying It\^{o} formula (Proposition \ref{3}), one has
\begin{align}
 &e^{\beta t}\hat{Y}_{t}^{2} + \beta \int_t^T e^{\beta s}\hat{Y}_{s}^{2} ds
     + 2\int_t^T e^{\beta s}\mathbb{D}_{s}^{H} \hat{Y}_{s}\hat{Z}_{s} ds
     +2\int_t^T e^{\beta s}\hat{Y}_{s}\hat{Z}_{s} dB_{s}^{H} \nonumber \\
 =&  2\int_t^T e^{\beta s}\hat{Y}_{s}
     E'[f(s,\eta_{s},y'_{s},z'_{s},y_{s},z_{s}) -
     f(s,\eta_{s},\overline{y}'_{s},\overline{z}'_{s},\overline{y}_{s},\overline{z}_{s})] ds. \label{25}
\end{align}
We know (see Hu and Peng \cite{Hu}, Maticiuc and Nie \cite{Maticiuc}) that $\mathbb{D}_{s}^{H} \hat{Y}_{s} = \frac{\hat{\sigma}_{s}}{\sigma_{s}} \hat{Z}_{s}$.
Moreover, by Remark 6 in Maticiuc and Nie \cite{Maticiuc}, there is a constant $M>0$ such that for every $t\in [0,T]$,
$\frac{t^{2H-1}}{M}\leq \frac{\hat{\sigma}_{t}}{\sigma_{t}}\leq M t^{2H-1}$.
So from (\ref{25}), note Proposition \ref{2}, we have
\begin{align}\label{35}
 &E\left(e^{\beta t}\hat{Y}_{t}^{2} + \beta \int_t^T e^{\beta s}\hat{Y}_{s}^{2} ds
     +  \frac{2}{M}\int_t^T e^{\beta s}s^{2H-1}\hat{Z}_{s}^{2} ds \right)\nonumber\\
 \leq&   2E\int_t^T e^{\beta s}\hat{Y}_{s}
     E'[f(s,\eta_{s},y'_{s},z'_{s},y_{s},z_{s}) -
     f(s,\eta_{s},\overline{y}'_{s},\overline{z}'_{s},\overline{y}_{s},\overline{z}_{s})] ds.
\end{align}
From assumption (H2), note (\ref{34}), we obtain
\begin{align}
 &E\left(e^{\beta t}\hat{Y}_{t}^{2} + \beta \int_t^T e^{\beta s}\hat{Y}_{s}^{2} ds
     +  \frac{2}{M}\int_t^T e^{\beta s}s^{2H-1}\hat{Z}_{s}^{2} ds \right)  \nonumber\\
\leq&  2C E\int_t^Te^{\beta s}|\hat{Y}_{s}|E'\big[|\hat{y}'_{s}|+|\hat{z}'_{s}| + |\hat{y}_{s}|+|\hat{z}_{s}|\big] ds  \nonumber \\
=&  2C E\int_t^Te^{\beta s}|\hat{Y}_{s}|\big[ E(|\hat{y}_{s}|+|\hat{z}_{s}|) + |\hat{y}_{s}|+|\hat{z}_{s}|\big] ds. \label{87}
\end{align}
Therefore by choosing $\beta\geq 1$, and using H\"{o}lder inequality and Jensen inequality, we get
\begin{align}
& E\left(e^{\beta t}\hat{Y}_{t}^{2} + \int_t^T e^{\beta s}\hat{Y}_{s}^{2} ds
     +  \frac{2}{M}\int_t^T e^{\beta s}s^{2H-1}\hat{Z}_{s}^{2} ds \right) \nonumber\\
\leq & 4C\int_t^T \big(e^{\beta s} E|\hat{Y}_{s}|^{2}\big)^{\frac{1}{2}}
\big[e^{\beta s} E(|\hat{y}_{s}| + |\hat{z}_{s}|)^{2}\big]^{\frac{1}{2}} ds. \label{83}
\end{align}
Denote $x(t)=(e^{\beta t} E|\hat{Y}_{t}|^{2})^{\frac{1}{2}}$. Then from (\ref{83}),
\begin{equation*}
  x(t)^{2}\leq  4C\int_t^T x(s)\bigg(\big[e^{\beta s}E(|\hat{y}_{s}|+|\hat{z}_{s}|)^{2}\big]^{\frac{1}{2}}\bigg) ds.
\end{equation*}
Applying Lemma 20 in Maticiuc and Nie \cite{Maticiuc} to the above inequality, one has
\begin{align*}
  x(t)\leq 2C\int_t^T \big[e^{\beta s} E(|\hat{y}_{s}| + |\hat{z}_{s}|)^{2}\big]^{\frac{1}{2}} ds.
\end{align*}
Therefore for $t\in [t_{n},T]$,
\begin{align*}
  x(t)^{2}\leq 4C^{2} \bigg(\int_{t_{n}}^T \big[e^{\beta s} E(|\hat{y}_{s}|^{2}
   + |\hat{z}_{s}|^{2})\big]^{\frac{1}{2}} ds\bigg)^{2}.
\end{align*}
Now we compute
\begin{align}
  \int_{t_{n}}^T x(s)^{2} ds
  \leq& 4C^{2}(T-t_{n}) \bigg(\int_{t_{n}}^T \big[e^{\beta s} E(|\hat{y}_{s}|^{2} + |\hat{z}_{s}|^{2})\big]^{\frac{1}{2}} ds\bigg)^{2}. \label{84}
\end{align}
From H\"{o}lder inequality,
\begin{align}
      &\bigg(\int_{t_{n}}^T \big[e^{\beta s} E(|\hat{y}_{s}|^{2}+|\hat{z}_{s}|^{2})\big]^{\frac{1}{2}}ds\bigg)^{2}\nonumber\\
 \leq& \bigg(\int_{t_{n}}^T \big[e^{\beta s} E|\hat{y}_{s}|^{2} \big]^{\frac{1}{2}} ds
      +\int_{t_{n}}^T \big[e^{\beta s} E|\hat{z}_{s}|^{2}\big]^{\frac{1}{2}}ds\bigg)^{2} \nonumber\\
 \leq& 2\bigg(\int_{t_{n}}^T \big[e^{\beta s} E|\hat{y}_{s}|^{2} \big]^{\frac{1}{2}} ds \bigg)^{2}
      +2\bigg(\int_{t_{n}}^T \big[\frac{1}{s^{2H-1}} \cdot e^{\beta s} s^{2H-1}E|\hat{z}_{s}|^{2}\big]^{\frac{1}{2}}ds\bigg)^{2} \nonumber\\
 \leq& 2(T-t_{n})\int_{t_{n}}^Te^{\beta s} E|\hat{y}_{s}|^{2} ds
      +\frac{2(T^{2-2H}-t_{n}^{2-2H})}{2-2H} \int_{t_{n}}^Te^{\beta s} s^{2H-1}E|\hat{z}_{s}|^{2} ds \nonumber\\
 \leq& \big[2(T-t_{n})+ \frac{T^{2-2H}-t_{n}^{2-2H}}{1-H}\big]
      E \int_{t_{n}}^{T} e^{\beta s}\big(|\hat{y}_{s}|^{2} + s^{2H-1}|\hat{z}_{s}|^{2}\big) ds.  \label{85}
\end{align}
Hence, from (\ref{84}) and (\ref{85}),
\begin{align}
  \int_{t_{n}}^T x(s)^{2} ds
  \leq& (T-t_{n})G\cdot
      E \int_{t_{n}}^{T} e^{\beta s}\big(|\hat{y}_{s}|^{2} + s^{2H-1}|\hat{z}_{s}|^{2}\big) ds,  \label{86}
\end{align}
where $G=4C^{2}\big[2(T-t_{n})+ \frac{T^{2-2H}-t_{n}^{2-2H}}{1-H}\big]$. Similarly,
\begin{align}
  \int_{t_{n}}^T \frac{1}{s^{2H-1}}x(s)^{2} ds
  \leq& \frac{T^{2-2H}-t_{n}^{2-2H}}{1-H}G\cdot
      E \int_{t_{n}}^{T} e^{\beta s}\big(|\hat{y}_{s}|^{2} + s^{2H-1}|\hat{z}_{s}|^{2}\big) ds.  \label{88}
\end{align}
From (\ref{87}) and the inequality $2ab\leq \frac{1}{\delta}a^{2} + \delta b^{2}$, where $\delta>0$ is a constant,
combining Jensen inequality, we deduce
\begin{align*}
 &E\left(\int_{t_{n}}^T e^{\beta s}\hat{Y}_{s}^{2} ds
     +  \frac{2}{M}\int_{t_{n}}^T e^{\beta s}s^{2H-1}\hat{Z}_{s}^{2} ds \right)\\
\leq&   2C E\int_{t_{n}}^T e^{\beta s} \bigg(\frac{1}{\delta}(1+\frac{1}{s^{2H-1}})|\hat{Y}_{s}|^{2}
        + \delta|\hat{y}_{s}|^{2} + \delta s^{2H-1}|\hat{z}_{s}|^{2}  \bigg) ds \nonumber \\
\leq&   \frac{2C}{\delta} E\int_{t_{n}}^T e^{\beta s}(1+\frac{1}{s^{2H-1}})|\hat{Y}_{s}|^{2}ds
        + 2C\delta E\int_{t_{n}}^T e^{\beta s} \bigg(|\hat{y}_{s}|^{2} +  s^{2H-1}|\hat{z}_{s}|^{2}  \bigg) ds.
\end{align*}
Applying the inequalities (\ref{86}) and (\ref{88}), and choosing $M\geq2$, one has
\begin{align*}
  E \int_{t_{n}}^{T} e^{\beta s}\big(|\hat{Y}_{s}|^{2} + s^{2H-1}|\hat{Z}_{s}|^{2}\big) ds
\leq \widetilde{G} E\int_{t_{n}}^{T} e^{\beta s}\big(|\hat{y}_{s}|^{2} + s^{2H-1}|\hat{z}_{s}|^{2}\big) ds,
\end{align*}
where
$$\widetilde{G}=\frac{CGM}{\delta}(T-t_{n}) + \frac{CGM}{\delta(1-H)}(T^{2-2H}-t_{n}^{2-2H}) + CM \delta.$$
Now, by choosing $\delta$ such that $ CM \delta<\frac{1}{4}$, and taking $n$ large enough such that
$$\frac{CGM}{\delta}(T-t_{n})<\frac{1}{4}, \ \ \frac{CGM}{\delta(1-H)}(T^{2-2H}-t_{n}^{2-2H}) <\frac{1}{4}.$$
Then we obtain
\begin{align*}
  E \int_{t_{n}}^{T} e^{\beta s}\big(|\hat{Y}_{s}|^{2} + s^{2H-1}|\hat{Z}_{s}|^{2}\big) ds
\leq \frac{3}{4} E\int_{t_{n}}^{T} e^{\beta s}\big(|\hat{y}_{s}|^{2} + s^{2H-1}|\hat{z}_{s}|^{2}\big) ds,
\end{align*}
Consequently, $I$ is a contraction on
$\widetilde{\mathcal{V}}_{[t_{n},T]} \times \widetilde{\mathcal{V}}^{H}_{[t_{n},T]}$.
Discussing as in the proof of Theorem 22 of Maticiuc and Nie \cite{Maticiuc},
we know BSDE (\ref{31}) admits a unique solution on $\widetilde{\mathcal{V}}_{[t_{n},T]} \times \widetilde{\mathcal{V}}_{[t_{n},T]}^{H}$.
Next, the procedure is to solve (\ref{31}) in $[t_{n-1},t_{n}].$
Repeating the above technique and discussion, we obtain that (\ref{31}) has a unique solution in
 $\widetilde{\mathcal{V}}_{[0,T]} \times \widetilde{\mathcal{V}}_{[0,T]}^{H}$.

Now we prove the estimate (\ref{20}).
Suppose $(Y,Z)$ is the solution of Eq. (\ref{31}).
Similarly to (\ref{35}), and from (H2), we obtain
\begin{align*}
 &E\left(e^{\beta t}Y_{t}^{2} + \beta \int_t^T e^{\beta s}Y_{s}^{2} ds
     +  \frac{2}{M}\int_t^T e^{\beta s}s^{2H-1}Z_{s}^{2} ds \right) \\
 \leq&  E\bigg(e^{\beta T}|g(\eta_{T})|^{2}
     +  2\int_t^T e^{\beta s}Y_{s} E'[f(s,\eta_{s},Y'_{s},Z'_{s},Y_{s},Z_{s})] ds\bigg) \nonumber\\
 \leq& E e^{\beta T}|g(\eta_{T})|^{2}+  E\int_t^T e^{\beta s} |f_{0}(s,\eta_{s})|^{2} ds \nonumber\\
     &+E\int_t^T \big(1+4C+\frac{4C^{2}M}{s^{2H-1}}\big)e^{\beta s} |Y_{s}|^{2} ds
     +   \frac{1}{M}E\int_t^T e^{\beta s} s^{2H-1}|Z_{s}|^{2} ds.
\end{align*}
Therefore
\begin{align}
 &E\left(e^{\beta t}Y_{t}^{2}
     +  \frac{1}{M}\int_t^T e^{\beta s}s^{2H-1}Z_{s}^{2} ds \right) \nonumber\\
 \leq& R \Theta(t,T,K) +
      E\int_t^T \big(1+4C+\frac{4C^{2}M}{s^{2H-1}}\big)e^{\beta s} |Y_{s}|^{2} ds. \label{22}
\end{align}
Then, by Gronwall inequality,
\begin{align*}
 Ee^{\beta t}Y_{t}^{2}
 \leq R \Theta(t,T,K)\exp\bigg\{ (1+4C)(T-t) + 4C^{2}M\frac{T^{2-2H}-t^{2-2H}}{2-2H} \bigg\}.
\end{align*}
Again from  (\ref{22}), we have
\begin{align*}
   E\int_t^T e^{\beta s}s^{2H-1}|Z_{s}|^{2} ds \leq R\Theta(t,T,K).
\end{align*}
Therefore the estimate (\ref{20}) is obtained.
This completes the proof.
\end{proof}

\subsection{The Second Method}

Here, we introduce another method to prove the existence and uniqueness of Eq. (\ref{31}).
It should be pointed out that this method is more convenient than the above one.
However, the price of doing so is that we should strengthen the condition of the coefficient $f$ with respect to $z$.

\begin{itemize}
  \item [(H3)]
  For the coefficient $f=f(t,x,y',z',y,z):[0,T]\times \mathbb{R}^{5}\rightarrow \mathbb{R}$ with
 $f\in C_{pol}^{0,1}([0,T]\times \mathbb{R}^{5})$,  there is a constant $C\geq 0$ such that,
     for every $t\in [0,T]$, $x,y_{1},y_{2},z_{1},z_{2},y'_{1},$ $y'_{2},z'_{1},z'_{2} \in \mathbb{R}$, we have
\begin{align}
      &|f(t,x,y'_{1},z'_{1},y_{1},z_{1}) - f(t,x,y'_{2},z'_{2},y_{2},z_{2})| \nonumber\\
      \leq& C\big(|y'_{1}-y'_{2}| + t^{H-\frac{1}{2}}|z'_{1}-z'_{2}| + |y_{1}-y_{2}| + t^{H-\frac{1}{2}}|z_{1}-z_{2}| \big).
      \nonumber
\end{align}
\end{itemize}

\begin{remark}\label{36}
Suppose $\alpha$ and $ \gamma$ are two square integrable, jointly measurable stochastic processes.
Then we can define for all $t\in [0,T]$, $x,y,z\in \mathbb{R}$,
\begin{equation*}
  f^{\alpha,\gamma}(t,x,y,z):= E'[f(t,x,\alpha'_{t},\gamma'_{t},y,z)]
   = \int_{\Omega} f(t,x,\alpha_{t}(\omega'),\gamma_{t}(\omega'),y,z) P (d\omega').
\end{equation*}
Indeed, due to the assumptions on the coefficient $f\in C_{pol}^{0,1}([0,T]\times \mathbb{R}^{5})$,
we have that $f^{\alpha,\gamma}\in C_{pol}^{0,1}([0,T]\times \mathbb{R}^{3})$.
In addition, with the same constant $C$ of assumption (H3), for every $t\in [0,T]$, $x,y_{1},y_{2},z_{1},z_{2} \in \mathbb{R}$, we have
\begin{equation*}
  |f^{\alpha,\gamma}(t,x,y_{1},z_{1}) - f^{\alpha,\gamma}(t,x,y_{2},z_{2})|\leq C\big(|y_{1}-y_{2}| + t^{H-\frac{1}{2}}|z_{1}-z_{2}| \big).
\end{equation*}
\end{remark}

\begin{theorem}\label{30}
Under (H1) and (H3), BSDE (\ref{31}) admits a unique solution.
\end{theorem}

We point out that in the proof of Theorem \ref{30}, the following result is used.

\begin{lemma}[Wen and Shi \cite{Wen}, Lemma 3.1]\label{37}
Suppose $g$ is a given differentiable function with polynomial growth, and $f(t,x)$ is a $C_{pol}^{0,1}$-continuous function.
Then the following BSDE:
\begin{equation*}\label{}
    Y_{t}=g(\eta_{T}) + \int_t^T f(s,\eta_{s}) ds - \int_t^T Z_{s} dB_{s}^{H}, \ \ t\in[0,T],
\end{equation*}
admits a unique solution
 $(Y_{\cdot},Z_{\cdot}) \in \widetilde{\mathcal{V}}_{[0,T]} \times \widetilde{\mathcal{V}}^{H}_{[0,T]}$.
 Moreover,
\begin{align}
    & E\left(e^{\beta t}|Y_{t}|^{2} + \frac{\beta}{2}\int_t^T e^{\beta s}|Y_{s}|^{2} ds + \frac{2}{M}\int_t^T s^{2H-1}e^{\beta s}|Z_{s}|^{2} ds\right) \nonumber\\
\leq& E\left(e^{\beta T}|g(\eta_{T})|^{2} + \frac{2}{\beta}\int_t^T e^{\beta s}|f(s,\eta_{s})|^{2} ds \right), \label{39}
\end{align}
where $M>0$ is a suitable constant and $\beta> 0$.
\end{lemma}

\begin{proof}[Proof of Theorem \ref{30}]
For every $(y_{\cdot},z_{\cdot}) \in \widetilde{\mathcal{V}}_{[0,T]} \times \widetilde{\mathcal{V}}^{H}_{[0,T]}$,
we consider BSDE as follows:
\begin{equation}\label{6}
   Y_{t} = g(\eta_{T}) + \int_t^T E'[f(s,\eta_{s},y'_{s},z'_{s},y_{s},z_{s})] ds - \int_t^T Z_{s} dB_{s}^{H},
    \ \ 0\leq t\leq T.
\end{equation}
From Lemma \ref{37}, Eq. (\ref{6}) admits the unique solution
$(Y_{\cdot},Z_{\cdot}) \in \widetilde{\mathcal{V}}_{[0,T]} \times \widetilde{\mathcal{V}}^{H}_{[0,T]}$.
Once again, define
$I:\widetilde{\mathcal{V}}_{[0,T]} \times \widetilde{\mathcal{V}}^{H}_{[0,T]}\longrightarrow \widetilde{\mathcal{V}}_{[0,T]} \times \widetilde{\mathcal{V}}^{H}_{[0,T]}$
such that $I[(y(\cdot),z(\cdot))]=((Y(\cdot),Z(\cdot))$.
Now we directly show that $I$ is a contraction mapping on
$\widetilde{\mathcal{V}}_{[0,T]} \times \widetilde{\mathcal{V}}^{H}_{[0,T]}$.
For two arbitrary elements $(y_{1}(\cdot),z_{1}(\cdot))$ and
 $(y_{2}(\cdot),z_{2}(\cdot))\in \widetilde{\mathcal{V}}_{[0,T]} \times \widetilde{\mathcal{V}}^{H}_{[0,T]}$,
set $(Y_{1}(\cdot),Z_{1}(\cdot))=I[(y_{1}(\cdot),z_{1}(\cdot))]$ and $(Y_{2}(\cdot),Z_{2}(\cdot))=I[(y_{2}(\cdot),z_{2}(\cdot))]$.
Denote
$ (\hat{y}(\cdot),\hat{z}(\cdot))=(y_{1}(\cdot)-y_{2}(\cdot),z_{1}(\cdot)-z_{2}(\cdot))$,
$(\hat{Y}(\cdot),\hat{Z}(\cdot))=(Y_{1}(\cdot)-Y_{2}(\cdot),Z_{1}(\cdot)-Z_{2}(\cdot)).$
Then, from the estimate (\ref{39}), we have
\begin{equation*}
\begin{split}
     &E\int_0^T e^{\beta s}\bigg(\frac{\beta}{2}|\hat{Y}(s)|^{2}  + \frac{2}{M}s^{2H-1}|\hat{Z}(s)|^{2} \bigg)ds\\
\leq& \frac{2}{\beta}E\int_0^T e^{\beta s}
      \bigg(E'\big[f(s,\eta_{s},y'_{1}(s),z'_{1}(s),y_{1}(s),z_{1}(s)) \\
    &\ \ \ \ \ \ \ \ \ \ \ \ \ \ \ \ \ \ \ \ \ - f(s,\eta_{s},y'_{2}(s),z'_{2}(s),y_{2}(s),z_{2}(s))\big]\bigg)^{2} ds.
\end{split}
\end{equation*}
From assumption (H3) we obtain
\begin{equation*}
 \begin{split}
  &E'\big[f(s,\eta_{s},y'_{1}(s),z'_{1}(s),y_{1}(s),z_{1}(s))  - f(s,\eta_{s},y'_{2}(s),z'_{2}(s),y_{2}(s),z_{2}(s))\big]\\
\leq& CE'\big[ |\hat{y}'(s)| + s^{H-\frac{1}{2}}|\hat{z}'(s)| + |\hat{y}(s)| + s^{H-\frac{1}{2}}|\hat{z}(s)|  \big]\\
  =& C\big(E|\hat{y}(s)| + s^{H-\frac{1}{2}}E|\hat{z}(s)| + |\hat{y}(s)| + s^{H-\frac{1}{2}}|\hat{z}(s)| \big).\\
 \end{split}
\end{equation*}
Therefore, by Jensen inequality and Fubini theorem, one has
\begin{equation*}
\begin{split}
     &E\int_0^T e^{\beta s}\bigg(\frac{\beta}{2}|\hat{Y}(s)|^{2}  + \frac{2}{M}s^{2H-1}|\hat{Z}(s)|^{2} \bigg)ds\\
\leq& \frac{2C^{2}}{\beta} E\int_0^T e^{\beta s}
      \bigg( E|\hat{y}(s)| + s^{H-\frac{1}{2}}E|\hat{z}(s)| + |\hat{y}(s)| + s^{H-\frac{1}{2}}|\hat{z}(s)| \bigg)^{2} ds\\
\leq& \frac{16C^{2}}{\beta} E\int_0^T e^{\beta s}
      \bigg( |\hat{y}(s)|^{2} + s^{2H-1}|\hat{z}(s)|^{2} \bigg) ds.\\
\end{split}
\end{equation*}
In other words,
\begin{align*}
 &E\int_0^T e^{\beta s}\bigg(\frac{M\beta}{4}|\hat{Y}(s)|^{2}  + s^{2H-1}|\hat{Z}(s)|^{2} \bigg)ds\\
\leq &\frac{8MC^{2}}{\beta} E\int_0^T e^{\beta s}
      \bigg( |\hat{y}(s)|^{2} + s^{2H-1}|\hat{z}(s)|^{2} \bigg) ds.
\end{align*}
Thus, by taking $\beta= 16MC^{2} + \frac{4}{M}$, we get
\begin{equation*}
     E\int_0^{T} e^{\beta s}\bigg(|\hat{Y}(s)|^{2} + s^{2H-1}|\hat{Z}(s)|^{2} \bigg)ds
\leq \frac{1}{2} E\int_0^{T} e^{\beta s} \bigg(|\hat{y}(s)|^{2} + s^{2H-1}|\hat{z}(s)|^{2} \bigg) ds.
\end{equation*}
Therefore,
$I$ is a contraction mapping on $\widetilde{\mathcal{V}}_{[0,T]} \times \widetilde{\mathcal{V}}^{H}_{[0,T]}$.
Consequently, Eq. (\ref{31}) admits a unique solution
  $(Y_{\cdot},Z_{\cdot})\in \widetilde{\mathcal{V}}_{[0,T]} \times \widetilde{\mathcal{V}}^{H}_{[0,T]}$.
\end{proof}

\begin{remark}
Now, we make a comparison for the above two methods.
It is easy to see that (H2) is weaker than (H3).
So from the point of view of the condition, the first method is better than the second one.
On the other hand, thanks to the concise proof, the second method is convenient than the first method.
So from this point of view, the second method is better.
\end{remark}

\section{Comparison theorem}

In this section, we study a comparison theorem of the fractional mean-field BSDEs of the following form:
\begin{equation}\label{41}
 Y_{t} = g(\eta_{T}) + \int_t^T E'[f(s,\eta_{s},Y_{s}',Y_{s},Z_{s})] ds - \int_t^T Z_{s} dB_{s}^{H}, \ \ 0\leq t\leq T.
\end{equation}
Under (H1) and (H3), it is easy to know that the above equation admits a unique solution.
Here we use (H3), not (H2), because it is more convenient for the proof of the following comparison theorem.
\begin{theorem}\label{40}
For $i=1,2$, suppose $g_{i}$ satisfies (H1), and $f_{i}(t,x,y',y,z)$ and $\partial_{y} f_{i}(t,x,y',y,z)$ satisfy (H3),
for each $(t,x,y',y,z)\in[0,T]\times \mathbb{R}^{4}$.
Moreover, assume $f_{1}$ is increasing in $y'$.
Then, if $g_{1}(x)\leq g_{2}(x)$ and $f_{1}(t,x,y',y,z)\leq f_{2}(t,x,y',y,z)$ for each $(t,x,y',y,z)\in[0,T]\times \mathbb{R}^{4}$,
it holds that $Y_{1}(t)\leq Y_{2}(t)$ almost surely.
\end{theorem}

\begin{proof}
For $i=1,2$, we define $f_{i}^{v}(s,x,y,z) = E'[f_{i}(s,x,v'_{s},y,z)]$.
By virtue of Remark \ref{36}, $f^{v}_{i}$ and $\partial_{y}f^{v}_{i}$ satisfy (H3), moreover,
$f^{v}_{1}\leq f^{v}_{2}$, and $f^{v}_{1}$ is increasing in $v$.

Let $\widetilde{Y}_{0}(\cdot)=Y_{2}(\cdot)$. We consider BSDE:
\begin{equation*}
 \widetilde{Y}_{1}(t)=g_{1}(\eta_{T}) + \int_t^T E'[f_{1}(s,\eta_{s},\widetilde{Y}'_{0}(s),\widetilde{Y}_{1}(s),\widetilde{Z}_{1}(s))] ds
   -\int_t^T \widetilde{Z}_{1}(s) dB_{s}^{H}, \ t\in[0,T].
\end{equation*}
By Theorem \ref{30}, the above equation admits a unique solution
$(\widetilde{Y}_{1}(\cdot),\widetilde{Z}_{1}(\cdot)) \in \widetilde{\mathcal{V}}_{[0,T]} \times \widetilde{\mathcal{V}}^{H}_{[0,T]}$.
Now since
\begin{equation*}
  \begin{cases}
   f_{1}^{\widetilde{Y}_{0}}(t,x,y,z) \leq  f_{2}^{\widetilde{Y}_{0}}(t,x,y,z), \ \ \forall (t,x,y,z)\in [0,T]\times \mathbb{R}^{3};\\
   g_{1}(x)\leq g_{2}(x), \ \ \forall x\in \mathbb{R},
  \end{cases}
\end{equation*}
from Theorem 12.3 of Hu et al. \cite{Hu1}, we deduce
\begin{equation*}
  \widetilde{Y}_{1}(t)\leq \widetilde{Y}_{0}(t)=Y_{2}(t), \ \ a.s.
\end{equation*}
Next, we consider the following BSDE:
\begin{equation*}
 \widetilde{Y}_{2}(t)=g_{1}(\eta_{T}) + \int_t^T E'[f_{1}(s,\eta_{s},\widetilde{Y}'_{1}(s),\widetilde{Y}_{2}(s),\widetilde{Z}_{2}(s))] ds
   -\int_t^T \widetilde{Z}_{2}(s) dB_{s}^{H}.
\end{equation*}
Let $(\widetilde{Y}_{2}(\cdot),\widetilde{Z}_{2}(\cdot)) \in \widetilde{\mathcal{V}}_{[0,T]} \times \widetilde{\mathcal{V}}^{H}_{[0,T]}$
be the unique solution of the above equation.
Thanks to that $f^{v}_{1}$ is increasing in $v$, one has
\begin{equation*}
f_{1}^{\widetilde{Y}_{1}}(t,x,y,z) \leq  f_{1}^{\widetilde{Y}_{0}}(t,x,y,z), \ \forall (t,x,y,z)\in [0,T]\times \mathbb{R}^{3}.\\
\end{equation*}
Therefore, similar to the above, we deduce
\begin{equation*}
   \widetilde{Y}_{2}(t)\leq \widetilde{Y}_{1}(t), \ \ a.s.
\end{equation*}
By induction, one can construct a sequence
$\{(\widetilde{Y}_{n}(\cdot),\widetilde{Z}_{n}(\cdot))\}_{n\geq 1} \subseteq \widetilde{\mathcal{V}}_{[0,T]} \times \widetilde{\mathcal{V}}^{H}_{[0,T]}$
such that for every $t\in[0,T]$,
\begin{equation*}
 \widetilde{Y}_{n}(t)=g_{1}(\eta_{T}) + \int_t^T E'[f_{1}(s,\eta_{s},\widetilde{Y}'_{n-1}(s),\widetilde{Y}_{n}(s),\widetilde{Z}_{n}(s))] ds
   -\int_t^T \widetilde{Z}_{n}(s) dB_{s}^{H}.
\end{equation*}
Similarly, we obtain
\begin{equation*}
  Y_{2}(t)= \widetilde{Y}_{0}(t)\geq \widetilde{Y}_{1}(t)\geq \widetilde{Y}_{2}(t)\geq
   \cdots \geq \widetilde{Y}_{n}(t)\geq \cdots,  \ \ a.s.
\end{equation*}
In the following, we show $\{(\widetilde{Y}_{n}(\cdot),\widetilde{Z}_{n}(\cdot))\}_{n\geq 1}$ is a Cauchy sequence. Denote
$\hat{Y}_{n}(t)=\widetilde{Y}_{n}(t)-\widetilde{Y}_{n-1}(t),$ and $\hat{Z}_{n}(t)=\widetilde{Z}_{n}(t)-\widetilde{Z}_{n-1}(t), n\geq 4$.
Then, from the estimate (\ref{39}), we have
\begin{equation*}
  \begin{split}
     & E\int_0^T e^{\beta s}\left(\frac{\beta}{2}|\hat{Y}_{n}(s)|^{2} +\frac{2}{M} s^{2H-1}|\hat{Z}_{n}(s)|^{2}\right) ds\\
 \leq& \frac{2}{\beta}E\int_0^T e^{\beta s}|f_{1}^{\widetilde{Y}_{n-1}}(s,\eta_{s},\widetilde{Y}_{n}(s),\widetilde{Z}_{n}(s))
      -f_{1}^{\widetilde{Y}_{n-2}}(s,\eta_{s},\widetilde{Y}_{n-1}(s),\widetilde{Z}_{n-1}(s))|^{2} ds.
  \end{split}
\end{equation*}
From assumption (H2), one has
\begin{equation*}
  \begin{split}
     & E\int_0^T e^{\beta s}\left(\frac{\beta}{2}|\hat{Y}_{n}(s)|^{2} +\frac{2}{M} s^{2H-1}|\hat{Z}_{n}(s)|^{2}\right) ds\\
 \leq& \frac{6C^{2}}{\beta}E\int_0^T e^{\beta s}\left( |\hat{Y}_{n}(s)|^{2} + s^{2H-1}|\hat{Z}_{n}(s)|^{2} + |\hat{Y}_{n-1}(s)|^{2} \right) ds.
  \end{split}
\end{equation*}
Let $\beta=12MC^{2}+\frac{4}{M}$, then we obtain
\begin{equation*}
  \begin{split}
     & E\int_0^T e^{\beta s}\left(|\hat{Y}_{n}(s)|^{2} + s^{2H-1}|\hat{Z}_{n}(s)|^{2}\right) ds\\
 \leq& \frac{1}{4}E\int_0^T e^{\beta s}\left( |\hat{Y}_{n}(s)|^{2} + s^{2H-1}|\hat{Z}_{n}(s)|^{2} + |\hat{Y}_{n-1}(s)|^{2} \right) ds.
  \end{split}
\end{equation*}
Hence
\begin{equation*}
  \begin{split}
     & E\int_0^T e^{\beta s}\left(|\hat{Y}_{n}(s)|^{2} + s^{2H-1}|\hat{Z}_{n}(s)|^{2}\right) ds\\
 \leq& \frac{1}{3}E\int_0^T e^{\beta s} |\hat{Y}_{n-1}(s)|^{2} ds\\
 \leq& \frac{1}{3}E\int_0^T e^{\beta s} \left(|\hat{Y}_{n-1}(s)|^{2} + s^{2H-1}|\hat{Z}_{n-1}(s)|^{2}\right) ds.
  \end{split}
\end{equation*}
Therefore
\begin{align*}
  & E\int_0^T e^{\beta s}\left(|\hat{Y}_{n}(s)|^{2} + s^{2H-1}|\hat{Z}_{n}(s)|^{2}\right) ds\\
 \leq& (\frac{1}{3})^{n-4}E\int_0^T e^{\beta s} \left(|\hat{Y}_{4}(s)|^{2} + s^{2H-1}|\hat{Z}_{4}(s)|^{2}\right) ds.
\end{align*}
So $(\hat{Y}_{n}(\cdot))_{n\geq 4}$ and $(\hat{Z}_{n}(\cdot))_{n\geq 4}$ are Cauchy sequences in
$\widetilde{\mathcal{V}}_{[0,T]}$ and  $\widetilde{\mathcal{V}}_{[0,T]}^{H}$, respectively.
Denote their limits by $\widetilde{Y}_{\cdot}$ and $\widetilde{Z}_{\cdot}$, respectively.
From Theorem \ref{30}, we have $\widetilde{Y}(t)= Y_{1}(t), \ \ a.s.$,
which deduce that
\begin{equation*}
 Y_{1}(t)\leq Y_{2}(t), \ \ a.s.
\end{equation*}
Therefore, our conclusion follows.
\end{proof}

\begin{example}
 Suppose we are facing with the following two mean-field BSDEs,
\begin{equation*}
 Y_{1}(t) = g_{1}(\eta_{T}) + \int_t^T [Y_{1}(s) + EY_{1}(s) + Z_{1}(s) -   1] ds - \int_t^T Z_{1}(s) dB_{s}^{H};
\end{equation*}
\begin{equation*}
 Y_{2}(t) = g_{2}(\eta_{T}) + \int_t^T [Y_{2}(s) + EY_{2}(s) + Z_{2}(s) + 1] ds - \int_t^T Z_{2}(s) dB_{s}^{H},
\end{equation*}
 where $t\in[0,T]$, $g_{1}$ and $g_{1}$ satisfy (H1) with $g_{1}(x)\leq g_{2}(x), \ \forall x\in \mathbb{R}$.
 Then, according to Theorem \ref{40}, one has
 \begin{equation*}
 Y_{1}(t)\leq Y_{2}(t), \ \ a.s.
\end{equation*}
\end{example}

\section{Connection with PDEs}

As an important application of the fractional mean-field BSDEs,
we connect the mean-field BSDEs driven by fBm with the following PDE:
\begin{equation}\label{51}
 \begin{cases}
 &v_{t}(t,x) +v_{x}(t,x)b_{t} + \frac{1}{2}v_{xx}(t,x)\widetilde{\sigma}_{t} \\
 &\ \ \ \ \ \ + E\big[f(t,x,v(t,\eta_{t}),v(t,x),v_{x}(t,x)\sigma_{t})\big]=0, \ \ \ (t,x)\in [0,T)\times \mathbb{R}; \\
 & v(T,x)=g(x), \ \ x\in \mathbb{R},
 \end{cases}
\end{equation}
where $\widetilde{\sigma}_{t}:=\frac{d}{dt}(\| \sigma\|_{t}^{2})$. From (\ref{34}), one has that
\begin{equation}\label{53}
  E\big[f(t,x,v(t,\eta_{t}),v(t,x),v_{x}(t,x)\sigma_{t})\big] = E'\big[f(t,x,v(t,\eta'_{t}),v(t,x),v_{x}(t,x)\sigma_{t})\big].
\end{equation}
It is easy to see PDE (\ref{51}) is a special case of Eq. (6.1) of Buckdahn et al. \cite{Buckdahn2}.
As presented in \cite{Buckdahn2}, PDE (\ref{51}) is nonlocal and under sufficient conditions,
it has a unique viscosity solution.
In other words, PDE (\ref{51}) is well-defined.

\begin{theorem}\label{10}
If PDE (\ref{51}) has a solution $v(t,x)$ with $v\in C^{1,2}([0,T]\times \mathbb{R})$,
then $(Y_{t},Z_{t}):=(v_{x}(t,\eta_{t}),$ $v_{x}(t,\eta_{t})\sigma_{t})$ satisfy the fractional mean-field BSDE:
\begin{equation}\label{52}
  Y_{t} = g(\eta_{T}) + \int_t^T E'\big[f(s,\eta_{s},Y'_{s},Y_{s},Z_{s})\big] ds - \int_t^T Z_{s} dB_{s}^{H}, \ \ 0\leq t\leq T.
\end{equation}
\end{theorem}

\begin{proof}
By applying
It\^{o} formula (Proposition \ref{4}), one has
\begin{equation*}
  \begin{split}
     dv(t,\eta_{t}) &= v_{t}(t,\eta_{t})dt
     +v_{x}(t,\eta_{t})b_{t}dt  +v_{x}(t,\eta_{t})\sigma_{t}dB_{t}^{H}
      + \frac{1}{2}v_{xx}(t,\eta_{t})\widetilde{\sigma}_{t} dt\\
      &= \left(v_{t}(t,\eta_{t}) +v_{x}(t,\eta_{t})b_{t} + \frac{1}{2}v_{xx}(t,\eta_{t})\widetilde{\sigma}_{t}\right) dt
         +v_{x}(t,\eta_{t})\sigma_{t} dB_{t}^{H}.
  \end{split}
\end{equation*}
Since $v$ satisfies PDE (\ref{51}), and note (\ref{53}), we have
\begin{equation*}
     dv(t,\eta_{t}) = E'\big[f(t,\eta_{t},v_{x}(t,\eta'_{t}),v(t,\eta_{t}),v_{x}(t,\eta_{t})\sigma_{t})\big] dt
         +v_{x}(t,\eta_{t})\sigma_{t} dB_{t}^{H}.
\end{equation*}
Therefore, $(Y_{t},Z_{t}):=(v_{x}(t,\eta_{t}),v_{x}(t,\eta_{t})\sigma_{t})$ satisfy the fractional mean-field BSDE (\ref{52}).
\end{proof}

\begin{remark}
Theorem \ref{10} establishes a relationship between the fractional mean-field BSDEs and PDEs,
 and extends Theorem 4.1 of Hu and Peng \cite{Hu} to mean-field circumstance.
 For the general result of the connection between the fractional mean-field BSDEs and PDEs, we will give some further studied
  in the future.
\end{remark}

\section{Conclusions}

The fractional mean-field BSDEs with Hurst parameter H $> \frac{1}{2}$ are studied in this article.
We first proposed two different methods to prove such equations admit a unique solution.
Then a comparison theorem for the solutions is obtained.
Finally, we established a relationship between this mean-field BSDE and a nonlocal PDE.
It should be pointed out that our results generalize part of the main results of Buckdahn et al. \cite{Buckdahn,Buckdahn2} to fractional calculus.
In the coming future researches, we would devote to establish the dual relation between fractional mean-field SDEs and BSDEs, and study the related stochastic optimal control problem.
The theory for the fractional mean-field BSDEs with  $H < \frac{1}{2}$ is anther goal.


\bibliographystyle{elsarticle-num}

\begin{thebibliography}{99}

\bibitem{Bender}
   C. Bender,
   Backward SDEs driven by Gaussian processes,
   Stochastic Process. Appl. 124 (2014) 2892-2916.

\bibitem{Biagini02}
   F. Biagini, Y. Hu, B. {\O}ksendal, A. Suleme,
   A stochastic maximum principle for processes driven by fractional Brownian motion,
   Stochastic Process. Appl. 100 (2002) 233-253.

\bibitem{Borkowska} K.J. Borkowska,
    Generalized BSDEs driven by fractional Brownian motion,
    Statist. Probab. Lett. 83 (2013) 805-811.

\bibitem{Buckdahn}
   R. Buckdahn, B. Djehiche, J. Li, S. Peng,
   Mean-field backward stochastic differential equations: A limit approach,
   Ann. Probab. 37 (2009) 1524-1565.

\bibitem{Buckdahn2}
   R. Buckdahn, J. Li, S. Peng,
   Mean-field backward stochastic differential equations and related partial differential equations,
   Stochastic Process. Appl. 119 (2009) 3133-3154.

\bibitem{Buckdahn2017}
   R. Buckdahn, J. Li, S. Peng, C. Rainer,
   Mean-field stochastic differential equations and associated PDEs,
   Ann. Probab. 45(2) (2017) 824-878.

\bibitem{Buckdahn3}
   R. Buckdahn, S. Jing,
   Mean-field SDE driven by a fractional Brownian motion and related stochastic control problem,
   (2016) ArXiv:1605.09488.

\bibitem{Decreusefond}
   L. Decreusefond, A.S. \"{U}st\"{u}nel,
   Stochastic analysis of the fractional Brownian motion,
   Potential Anal. 10 (1999) 177-214.

\bibitem{Peng2}
   N. El Karoui, S. Peng, M.C. Quenez,
   Backward stochastic differential equations in finance,
   Math. Finance 7 (1997) 1-71.

\bibitem{Hu3}
   Y. Hu,
   Integral transformations and anticipative calculus for fractional Brownian motions,
   Mem. Amer. Math. Soc. 175 (2005) no. 825.

\bibitem{Han13} Y. Han, Y. Hu, J. Song,
   Maximum principle for general controlled systems driven by fractional Brownian motions,
   Appl. Math. Optim. 67 (2013) 279-322.

\bibitem{Hu1}
   Y. Hu, D. Ocone, J. Song,
   Some results on backward stochastic differential equations driven by fractional Brownian motions,
   Stoch. Anal. Appl. Finance (2012) 225-242.

\bibitem{Hu} Y. Hu, S. Peng,
   Backward stochastic differential equation driven by fractional Brownian motion,
   SIAM J. Control Optim. 48 (2009) 1675-1700.

\bibitem{Lasry}
   J.M. Lasry, P.L. Lions,
   Mean field games,
   Japan. J. Math. 2 (2007) 229-260.

\bibitem{Maticiuc}
   L. Maticiuc, T. Nie,
   Fractional backward stochastic differential equations and fractional backward variational inequalities,
   J. Theory Probab. 28 (2015) 337-395.

\bibitem{Nualart}
   D. Nualart,
   The Malliavin Calculus and Related Topics (Second Edition),
   Springer, 2006.

\bibitem{Peng}
   E. Pardoux, S. Peng,
   Adapted solution of a backward stochastic differential equation,
   Systems Control Lett. 4 (1990) 55-61.

\bibitem{Peng92}
   E. Pardoux, S. Peng,
   Backward SDEs and quasi-linear PDEs,
   Lecture Notes in Control and Inform Sci. 176 (1992) 200-217.

\bibitem{Wen} J. Wen, Y. Shi,
  Anticipative backward stochastic differential equations driven by fractional Brownian motion,
  Statist. Probab. Lett.  122 (2017) 118-127.

  \bibitem{Yong5}
J. Yong, X. Zhou,
Stochastic Controls: Hamiltonian Systems and HJB Equations,
Springer-Verlag, New York, 1999.

\end{thebibliography}

\end{document}